\documentclass[a4paper,12pt]{amsart}
\usepackage{amsmath}
\usepackage{amsfonts}
\usepackage{amssymb,multicol}
\usepackage{multicol} 
 
\newtheorem{theorem}{Theorem}
\theoremstyle{plain}

\newtheorem{lemma}{Lemma}

\newtheorem{remark}{Remark}

\numberwithin{equation}{section}
\numberwithin{lemma}{section}
\numberwithin{theorem}{section}
\numberwithin{corollary}{section}
\numberwithin{proposition}{section}

\usepackage{hyperref}  
\hypersetup{colorlinks=true,
	linkcolor= blue,
	filecolor=magenta,
	citecolor= blue,
	urlcolor=cyan,}
\title{  on the lower Lie nilpotency index  of a group algebra  }
\author{Meena Sahai }
\address {Department of Mathematics and Astronomy\\
 University of Lucknow, Lucknow, U.P. 226007, India}
	\email{meena\_sahai@hotmail.com}
	\author{Bhagwat Sharan}
	\curraddr{Department of Mathematics and Astronomy, University of Lucknow,
		Lucknow, U.P. 226007, India}
	\email{bhagwat\_sharan@hotmail.com}
	\subjclass[2010]{Primary 16S34 ; Secondary 17B30.}
	\keywords{ Group algebras, Lie nilpotency indices, Lie dimension subgroups.}	
\begin{document}
\maketitle
\begin{abstract}
In this article, we show that if $KG$ is Lie nilpotent group algebra of a group $G$ over a field $K$ of characteristic $p>0$, then  $t_{L}(KG)=k$ if and only if $t^{L}(KG)=k$, for $k\in\{5p-3, 6p-4\}$, where $t_{L}(KG)$ and $t^{L}(KG)$ are the lower and the upper Lie nilpotency indices of $KG$, respectively.
\end{abstract}
\section{Introduction}
Let $R$ an associative ring  with identity. $R$ can be treated as a Lie ring  under the Lie product $[ x, y] = xy-yx$ for any $x, y\in R$. Set $R^{[1]} = R^{(1)}=R$ and $[x_{1},\cdots, x_{n}] = [[x_{1},\cdots,x_{n-1}]$, $x_{n}]$, where $x_{i}\in R$. For $n \geq 2$, let $R^{[n]} $ be the associative ideal of $R$ generated by  all the Lie commutators $[x_{1},\cdots, x_{n}]$, where $x_{i}\in R$ and let $R^{(n)}$ be the associative ideal of $R$ generated by all the Lie commutators $[x,y]$, where $x \in R^{(n-1)}$, $y\in R$. The ring $R$ is  Lie nilpotent (respectively, strongly Lie nilpotent) if  $R^{[n]} = 0$ $(R^{(n)}=0)$ for some positive integer $n$. The lower (upper) Lie nilpotency index of $R$ denoted by $t_{L}{(R)}$   $(t^{L}{(R)})$  is the minimal positive integer $n$ such that $R^{[n]} = 0$ $(R^{(n)}=0)$, respectively. Clearly, $t_{L}{(R)}\leq t^{L}(R)$. It is important to note that there is an example of an algebra $R$ of characteristic two which is not strongly Lie nilpotent but for which $t_{L}{(R)}=3$, see \cite{GL}.

 Let $R=KG$ be the group algebra of an arbitrary group $G$ over a field $K$. The augmentation ideal $\Delta_{K}(G)$ of $KG$ is nilpotent if and only if $K$ has characteristic $p>0$ and $G$ is a finite $p$-group. Let $t(G)$ denote  the nilpotency index of $\Delta_{K}(G)$. A non-commutative modular group algebra  $KG$ is Lie nilpotent if and only if $K$ has characteristic $p > 0$, $G$ is nilpotent and its commutator subgroup $G{^{\prime }}$ is a finite $p$-group, see \cite{PPS}. In \cite[pp. 46, 48]{Pa} explicit expressions are given for the Lie dimension subgroups $D_{(m),K}(G)$,$m\geq1
$, of $G$ over $K$, if  characteristic of $K=p> 0$,
\begin{equation*}
	D_{(m),K}(G) = \underset{\left( i-1\right) p^{j}\geq m-1}{\Pi }\gamma_{i}\left( G\right)
	^{p^{j}}.
\end{equation*}
According to \cite{Sh1}, if $KG$ is Lie nilpotent such that $|G'|=p^{n}$, then  the upper Lie nilpotency index $t^{L}(KG)$ $=$ $2+$ $(p-1)\sum\limits_{m\geq1}md_{(m+1)}$, where $p^{d_{(m)}} = | D_{(m),K}(G) : D_{(m+1),K}(G)|$ for every $m\geq 2$. Clearly, $\sum\limits_{m\geq2} d_{(m)}=n$. If $p>3$, then $t_{L}(KG$ $=$ $t^{L}(KG)$, see \cite[Theorem~1]{BP}. 

If $G$ is a nonabelian nilpotent group with $|G'|=p^{n}$, then by \cite{SB,SS}, $p+1\leq t_{L}(KG)\leq t^{L}(KG)\leq |G'|+1$. In \cite{SB}, it is proved that $t_{L}(KG)=p+1$ if and only if $t^{L}(KG)=p+1$. Lie nilpotent group algebras with Lie nilpotency index $2p$, $3p-1$ or $4p-2$ have been investigated in \cite{Sa}. Again in all these cases $t_{L}(KG)=t^{L}(KG)$. A complete description of Lie nilpotent modular group algebras $KG$ with Lie nilpotency index at most 8 is given in \cite{CS,msbs5}. More results on Lie nilpotency indices of modular group algebras are given in \cite{ Bo,BJS,BSp,BS,msbs1,msbs2,msbs3,Sh3}.

In this article, we show that if $KG$ is Lie nilpotent, then  $t_{L}(KG)=k$ if and only if $t^{L}(KG)=k$, for $k\in\{5p-3, 6p-4\}$.

We use the standard natations: $(C_{n})^{k}$ is the $k$-times direct product of the cyclic group $C_{n}$ of order $n$, $\zeta(G)$ is the center of a group $G$. All the group commutators are left normed  and $(x, y)= x^{-1}y^{-1}xy$ for all $x,y\in G$. The $i$th term of the lower central series of $G$ is $\gamma_{i}(G)$ and $G'=\gamma_{2}(G)$ is the commutator subgroup of $G$. 

We shall frequently use the following results from \cite{BKu,GL} without mentioning:
\begin{enumerate} 
	\item $R^{[m]}R^{[n]}\subseteq R^{[m + n - 2]}$. 
	\item $\gamma_{m}(U(R))\subseteq  1 + R^{[m]}$.
	\item $((x, y)-1)^{k}R^{[m]}\in  R^{[m+k]}$, for all $x,y\in U(R)$ and $k>1$.
\end{enumerate}	

\begin{lemma}\cite{msbs5}{\label{L2_6}}
Let $R$ be an associative ring with unity and let $k\geq1$. Then
\begin{enumerate} 
	\item $(R^{[3]})^{2k}\subseteq R^{[3k + 2]}$.
	\item $(R^{[3]})^{2k+1}\subseteq R^{[3k + 3]}$.
    \item $(R^{[3]})^{k}\subseteq R^{[2k + 1]}$, if 3 is a unit in $R$.
\end{enumerate}
 \end{lemma}
\begin{theorem}\cite{msbs5} \label{thm_6}
	Let $K$ be a field of characteristics $p>0$ and let $G$ be a nilpotent group such that $G'$ is a finite abelian $p$-group with the invariants $(p^{m_{1}}, p^{m_{2}}, \cdots, p^{m_{s}})$ and $\big\vert\gamma_{3}(G)G'^{p}/G'^{p}\big\vert = p^{r}$. Then the following statements hold:
\begin{enumerate} 
\item $t_{L}(KG)\geq t(G') + r+1$, if $p=3$; 
	\item $t_{L}(KG)\geq t(G') + r(p-1) +1$ , if $p\neq 3$.
\end{enumerate}
\textrm{where} $t(G')= 1+\sum\limits_{i=1}^{s}(p^{m_{i}} - 1)$.
\end{theorem}

\section{Main Results}
\begin{lemma}\cite[pp 33]{BKu}\label{L4_6}
Let $R$ be an associative ring with unity and $m\geq1$. Then for all $x,y\in U(R)$, $((x,y,y)-1)R^{[m]}\in R^{[m+2]}$.	
\end{lemma}
\begin{proof} We observe that
\begin{align*}
((x,y,y)-1) &= (x,y)^{-1}y^{-1}[(x,y),y] \\
   &= (x,y)^{-1}y^{-1}\Big(x^{-1}y^{-1}[x,y,y] + [x^{-1}y^{-1},y][x,y]\Big)\\
   &=(x,y)^{-1}y^{-1}\Big(x^{-1}y^{-1}[x,y,y] + x^{-1}[x,y][x,y, x^{-1}y^{-1}]\\
    &\qquad\qquad\qquad\qquad\qquad\qquad\qquad\qquad-x^{-1}[x,y]^{2}x^{-1}y^{-1}\Big).
\end{align*}	
By \cite{LS}, we get $((x,y,y)-1)R^{[m]}\in R^{[m+2]}$. 

\end{proof}
\begin{lemma}\label{L44_6}
Let $R$ be an associative ring with unity and let $2\in U(R)$. Then for all $x,y\in U(R)$, $((x,y,y,y)-1)^{2}\in R^{[7]}$.	
\end{lemma}
\begin{proof}
Let $t_{1}= (x,y)$, $t_{2}= (x,y,y)$ and $t_{3}= (x,y,y,y)$. Then $[t_{1},y]= [x^{-1}y^{-1},x,y]y$ and $[t_{1}^{-1},y]= y^{-1}[x^{-1}y,x,y]$. We  observe that
\begin{align*}
t_{3}-1 &= t_{2}^{-1}y^{-1}[(x,y,y),y] \\
&=t_{2}^{-1}y^{-1}[t_{1}^{-1}y^{-1}[t_{1},y],y]\\
&= t_{2}^{-1}y^{-1}\Big(t_{1}^{-1}y^{-1} [[t_{1},y],y]+ [t_{1}^{-1},y]y^{-1}[t_{1},y] \Big)\\
&= t_{2}^{-1}y^{-1}\Big(t_{1}^{-1}y^{-1} [x^{-1}y^{-1},x,y,y]y + y^{-1}[x^{-1}y,x,y]y^{-1}[x^{-1}y^{-1},x,y]y \Big)\\
& =\alpha + \beta,
\end{align*}
where  $\beta= t_{2}^{-1}y^{-2}[x^{-1}y,x,y]y^{-1}[x^{-1}y^{-1},x,y]y$ and $\alpha= t_{2}^{-1}y^{-1}t_{1}^{-1}y^{-1} [x^{-1}y^{-1},x,y,y]y$. By Lemma~\ref{L2_6}, $\beta\in R^{[5]}$ and
by \cite[Lemma 2.2]{SS}, $\alpha^{2}\in R^{[7]}$. Also $\alpha\beta,\beta\alpha\in R^{[7]}$ and $\beta^{2}\in R^{[8]}$. Hence $(t_{3}-1)^{2}\in R^{[7]}$.
\end{proof}
\begin{lemma} \label{L5_6} 
Let $G$ be a nonabelian group and let $R$ be a commutative ring such that
 $RG^{[n]} = 0$, $n\geq4$. Let  $G^{\prime}$ be a finite $p$-group for some prime  $p \geq 3$ and let  $m_{i}$ be the rank of the finite abelian group $\gamma_{i}(G)/\gamma_{i+1}(G)$, $i\geq2$. Then 
	\begin{equation*}
	m_{2} + \frac{3}{2}m_{3} + 2m_{4} + 3m_{5} + \cdots + (c-2)m_{c} \leq \frac{n-3}{p-1},
	\end{equation*}	
where $c$ is the class of $G$.	
\end{lemma}
\begin{proof}
	Let  $z_{ij}$, $j= 1,\cdots, m_{i}$, be the $m_{i}$ independent generators of  $\gamma_{i}(G)/\gamma_{i+1}(G)$. Then 
	\begin{equation*} 
	x=  \prod_{j=1}^{m_{2}}(z_{2j}-1)^{p-1} \prod_{j=1}^{m_{3}}(z_{3j}-1)^{p-1}\cdots\prod_{j=1}^{m_{c}}(z_{cj}-1)^{p-1}\neq 0.
	\end{equation*}	
	By Lemma~\ref{L2_6} $x\in RG^{[k]}$, where $k= 2+ m_{2}(p-1) + \frac{3}{2}m_{3}(p-1) + 2m_{4}(p-1) + 3m_{5}(p-1) + \cdots + (c-2)m_{c}(p-1)$  and $k\leq n-1$ yields the bound.
\end{proof}

\begin{lemma}\label{Lemma 2.4}
	Let $KG$ be a Lie nilpotent group algebra of a nonabelian group $G$ over a field $K$ of characteristic $p>0$ such that $t_{L}(KG) =11$. Then $p=2$.
\end{lemma}
\begin{proof}
If $p\geq 5$, then $t_{L}(KG)=t^{L}(KG)$ is always even by \cite[Theorem~1]{BP}. So let $p=3$ and $t_{L}(KG) =11$. Then $|G'|\geq 3^{3}$. Suppose that $x_{i}\in G$ such that $(x_{1},x_{2},\cdots,x_{i})\in\gamma_{i}(G)\backslash  \gamma_{i+1}(G)$, $i\geq 2$. Then by Lemma~\ref{L2_6}, $((x_{1},x_{2})-1)^{2}((x_{1},x_{2},x_{3})-1)^{2}((x_{1},x_{2},x_{3},x_{4})-1)^{2}\in KG^{[11]}=0$. So $\gamma_{4}(G)=1$ and $G'$ is abelian. Let  $(3^{m_{1}}, 3^{m_{2}}, \cdots, 3^{m_{s}})$ be the invariants of $G'$ and let $\big\vert\gamma_{3}(G)G'^{3}/G'^{3}\big\vert = 3^{r}$. Then by 
Theorem~\ref{thm_6}, $r+t\leq 9 $, where $ t= t(G^{\prime})-1 = \sum\limits_{i=1}^{s}(3^{m_{i}}-1)$ is an even number. Clearly  $t \geq 6$. We have the following two cases:

{\bf{Case 1.}}	If $t=8$, then $ G'\cong C_{9}$ or $(C_{3})^{4}$ and $r\leq1$.
If $G'\cong C_{9}$, then $t_{L}(KG) =10$, so $ G'\cong (C_{3})^{4}$. Now if $r=0$, then $\gamma_{3}(G)\subseteq G'^{3}$. But then by \cite[Theorem 3.2]{BKu}, $t_{L}(KG)=10$, a contradiction. If $r=1$, then  $\gamma_{3}(G)\cong C_{3}$. Let $b_{1},b_{2},b_{3}$ be the independent generators of $G'/\gamma_{3}(G)$  and let $\gamma_{3}(G)=\langle c_{1}\rangle$. Then $(b_{1}-1)^{2}(b_{2}-1)^{2}(b_{3}-1)^{2}(c_{1}-1)^{2}\in KG^{[11]}=0$, which is a contradiction.

{\bf{Case 2.}} If $t=6$, then $G'\cong (C_{3})^{3}$ and $r\leq3$. Clearly $r\neq3$. If $r\leq 1$, then $t_{L}(KG)\leq t^{L}(KG)\leq10$. So   $r=2$ and  $\gamma_{3}(G)\cong C_{3}\times C_{3}$.  But then as in   \cite{Bo}, by using  \cite{Bla58}, there exist $x,y,z\in G$ such that $a= (x,y), b=(x,y,y), c= (x,y,z)$, $G'= \langle a,b,c\rangle$ and $\gamma_{3}(G)= \langle b,c\rangle$.
Therefore by  Lemmas~\ref{L2_6} and \ref{L4_6}, $(a-1)^{2}(b-1)^{2}(c-1)^{2}\in KG^{[11]}=0$,
which is a contradiction. 
\end{proof}
\begin{lemma}\label{Lemma 2.5}
Let $KG$ be a Lie nilpotent group algebra of a nonabelian group $G$ over a field $K$ of characteristic $p>0$ such that $t_{L}(KG) =13$. Then $p=2$.
\end{lemma}
\begin{proof} As in the previous Lemma, $p< 5$. Let $p=3$ and $t_{L}(KG) =13$.	Then $G$ is nilpotent of class at most 4 by \cite{SS} and $|G'|\geq 3^{3}$. Thus $\gamma_{5}(G)=1$ and $\gamma_{3}(G)\subseteq\zeta(G')$. The exponent of $\gamma_{3}(G)$ is at most 3 because if $a_{1}, a_{2}, a_{3}\in G$, then $((a_{1}, a_{2}, a_{3})-1)^{8}\in (KG^{[3]})^{8}\subseteq KG^{[14]}=0$. Also by Lemma~\ref{L5_6}, $m_{2}+\frac{3}{2}m_{3}+2m_{4}\leq 5$. Clearly $m_{4}\leq 1$ and $m_{3}\leq 2$. Let $m_{4}=1$. Then $m_{2}=m_{3}=1$, $\gamma_{3}(G)\cong C_{3}\times C_{3}$ and $\gamma_{4}(G)\cong C_{3}$. Let $a,b$ and $c$ be the independent generators of $G'/\gamma_{3}(G)$, $\gamma_{3}(G)/\gamma_{4}(G)$ and   $\gamma_{4}(G)$ respectively.
If $|G'|\geq 3^{4}$, then $o(a)\geq9$ and  $(a-1)^{8}(b-1)^{2}(c-1)^{2}\in KG^{[17]}=0$, a contradiction. So $|G'|=3^{3}$ and  $G'$ is abelian. If $m_{4}=0$, then obviously $G'$ is abelian. Let  $(3^{m_{1}}, 3^{m_{2}}, \cdots, 3^{m_{s}})$ be the invariants of $G'$. Then by Theorem~\ref{thm_6}, $r+t\leq 11 $, where $ t= \sum\limits_{i=1}^{s}(3^{m_{i}}-1)$ is an even number and $\big\vert\gamma_{3}(G)G'^{3}/G'^{3}\big\vert = 3^{r}$. It is clear that  $t \geq 6$. We have the following cases: 

{\bf Case 1}: Let $t=10$. Then $G'\cong C_{9}\times C_{3}$ or $(C_{3})^{5}$ and $r\leq 1$. If $r=0$, then $\gamma_{3}(G)\subseteq G'^{3}$ and by \cite[Theorem 3.2]{BKu}, $t_{L}(KG)=12$, a contradiction.
If $r=1$ and $G'\cong (C_{3})^{5}$, then $\gamma_{3}(G)\cong C_{3}$. But then by \cite[Lemma~2.7]{Sa}, $t_{L}(KG) =14$. So $r=1$,  $G'\cong C_{9}\times C_{3}$ and $|\gamma_{3}(G)G'^{3}|=9$. Thus either $\gamma_{3}(G)\cap G'^{3}=1$, $\gamma_{3}(G)\cong C_{3}$ or $G'^{3}\subseteq\gamma_{3}(G)\cong C_{3}\times C_{3}$. If $\gamma_{3}(G)\cong C_{3}$ and $\gamma_{3}(G)\cap G'^{3}=1$, then let $G'/\gamma_{3}(G)=\langle b\gamma_{3}(G)\rangle$ and $\gamma_{3}(G)=\langle c\rangle$. Now $0\neq(b-1)^{8}(c-1)^{2}\in KG^{[13]}$, a contradiction. If $G'^{3}\subseteq\gamma_{3}(G)\cong C_{3}\times C_{3}$, then let $G'/\gamma_{3}(G)=\langle b\gamma_{3}(G)\rangle$, $b^{3}\in \gamma_{3}(G)$ and $\gamma_{3}(G)=\langle b^{3}\rangle\times \langle c\rangle$.  Again $0\neq(b-1)^{2}(b^{3}-1)^{2}(c-1)^{2}=(b-1)^{8}(c-1)^{2}\in  KG^{[13]}$, a contradiction.
 
{\bf Case 2}: Let $t=8$. Then $G'\cong C_{9}$ or $(C_{3})^{4}$ and $r\leq 3$.
Clearly $G'\cong C_{9}$ is not  possible. So $G'\cong (C_{3})^{4}$ and $r=m_{3}\leq2$. If $r\leq 1$, then $t_{L}(KG)\leq t^{L}(KG)\leq12$. So $r=2$,  $\gamma_{3}(G)\cong C_{3}\times C_{3}$ and $\gamma_{4}(G)=1$. Then by \cite{Bo}, there exist $x,y,z\in G$ such that $b_{1}=(x,y), ~b_{2}=(x,z), c_{1}=(x,y,y), ~c_{2}=(x,y,z)$, $G'=\langle b_{1}, b_{2},c_{1},c_{2}\rangle$ and  $\gamma_{3}(G)= \langle c_{1},c_{2}\rangle$. Therefore by  Lemmas~\ref{L2_6} and \ref{L4_6}, $(b_{1}-1)^{2}(b_{2}-1)^{2}(c_{1}-1)^{2}(c_{2}-1)^{2}\in KG^{[13]}=0$,
a contradiction. 

{\bf{Case 3.}} If $t=6$, then $G'\cong (C_{3})^{3}$ and $r\leq2$. For $r\leq 1$,  $t_{L}(KG)\leq t^{L}(KG)\leq 10$.  So $r=2$ and  $D_{(3),K}(G)=\gamma_{3}(G)\cong C_{3}\times C_{3}$. If $\gamma_{4}(G)=1$, then $t_{L}(KG)\leq t^{L}(KG)=12$. If $\gamma_{4}(G)\cong C_{3}$, then as in \cite{Bo} by using  \cite{Bla58}, there exist $x,y\in G$ such that $a= (x,y),~ b=(x,y,y),~ c= (x,y,y,y)$, $G'= \langle a,b,c\rangle$, $\gamma_{3}(G)= \langle b,c\rangle$ and $\gamma_{4}(G)= \langle c\rangle$.
Therefore by  Lemmas~\ref{L4_6} and \ref{L44_6}, $(a-1)^{2}(b-1)^{2}(c-1)^{2}\in KG^{[13]}=0$, a contradiction.
\end{proof}

\begin{theorem}\label{thm1_6} 
Let $G$ be a group and let $K$ be a field of characteristic $p>0$ such that $KG$ Lie nilpotent. Then $t_{L}(KG)=5p-3$ if and only if $t^{L}(KG)=5p-3$.
\end{theorem}
\begin{proof}
By \cite{BP}, we have to discuss only $p=2$ and $3$. For  $p=2$, $t_{L}(KG)=5p-3$ if and only if $t^{L}(KG)=5p-3$, by \cite{msbs5}. Let  $p=3$ and $t_{L}(KG)=12$. As in the previous Lemma, $G'$ is abelian and $|G'|\geq 3^{3}$. Also if $|G'|\geq 3^{4}$, then $\gamma_{4}(G) = 1$. Let  $(3^{m_{1}}, 3^{m_{2}}, \cdots, 3^{m_{s}})$ be the invariants of $G'$ and $\big\vert\gamma_{3}(G)G'^{3}/G'^{3}\big\vert = 3^{r}$. Then by Theorem~\ref{thm_6}, $ r+t\leq 10$,
 where $ t= \sum\limits_{i=1}^{s}(3^{m_{i}}-1)$ is an even number. Clearly, $t\geq 6$. We have the following cases:

 {\bf{Case 1.}}	If $t=10$, then $r=0$. Thus $G'\cong C_{9}\times C_{3}$ or $(C_{3})^{5}$ and $\gamma_{3}(G)\subseteq G'^{3}$. Hence  $t_{L}(KG)= t^{L}(KG)$ by \cite[Theorem 3.2]{BKu}. 

{\bf{Case 2.}} If $t=8$, then $G'\cong C_{9}$ or $(C_{3})^{4}$ and $r\leq2$. If $G'\cong C_{9}$, then $t_{L}(KG)=10$. So $G'\cong (C_{3})^{4}$. If $r=0$, then by \cite[Theorem 3.2]{BKu}, $t_{L}(KG)=10$.  If $r=1$, then  $D_{(3),K}(G)=\gamma_{3}(G)\cong C_{3}$. Hence $d_{(2)}=3$, $d_{(3)}=1$ and $t^{L}(KG)=12$. If $r=2$, then $\gamma_{3}(G)\cong C_{3}\times C_{3}$. By the same arguments as in the Case 2 of the previous Lemma, we get a contradiction.

{\bf{Case 3.}} If $t=6$, then $G'\cong (C_{3})^3$. Clearly $r\leq 2$. If  $r\leq1$, then $t_{L}(KG)\leq t^{L}(KG)\leq10$. So $r=2$ and   $D_{(3),K}(G)=\gamma_{3}(G)\cong C_{3}\times C_{3}$.   If $\gamma_{4}(G)=1$, then  $d_{(2)}=1$, $d_{(3)}=2$ and  $t^{L}(KG)=12$. The possibility that  $\gamma_{4}(G)\cong C_{3}$ can be dismissed using the same arguments as in the Case 3 of the previous Lemma.

Conversely, if $p=3$ and $t^{L}(KG)= 12$, then $t_{L}(KG)\leq 12$. Now $t_{L}(KG)\neq9$ and $11$ by  \cite{Sa} and Lemma~\ref{Lemma 2.4}. If $t_{L}(KG)=10=4p-2$, then $t^{L}(KG)=10$ by \cite{Sa}. Also $t_{L}(KG)\leq 8$ yields $t^{L}(KG)=t_{L}(KG)\leq 8$, by \cite{CS,msbs5}. So $t_{L}(KG)=12$.
\end{proof}
\begin{theorem} \label{thm2_6}
Let $KG$ be a Lie nilpotent group algebra of a group $G$ over a field $K$ of characteristic $p>0$. Then $t_{L}(KG)=6p-4$ if and only if $t^{L}(KG)=6p-4$. 
\end{theorem}
\begin{proof}
In view of \cite{BP, msbs5}, we have to consider the case $p=3$ only.
Let  $p=3$ and $t_{L}(KG)=14$. Then $G$ is nilpotent of class at most 5 by \cite{SS}. Thus $\gamma_{6}(G) =1$ and  $\gamma_{3}(G)$ is an abelian group. Also $|G'|\geq 3^{3}$ and the   exponent of $\gamma_{3}(G)$ is at most 3,  as for  $a_{1}, a_{2}, a_{3}\in G$,  $((a_{1}, a_{2}, a_{3})-1)^{8}\in (KG^{[3]})^{8}\subseteq KG^{[14]}$.  By Lemma~\ref{L5_6},
\begin{align*}
 m_{2}+\frac{3}{2}m_{3}+2m_{4}+3m_{5}\leq \frac{11}{2}.
\end{align*}
So $m_{5}\leq1$. If $m_{5}=1$, then we can not have $m_{2},m_{3},m_{4}\neq 0$. Thus $m_{5}=0$, $\gamma_{5}(G)=1$ and $\gamma_{3}(G)\subseteq\zeta(G')$. Clearly $m_{4}\leq1$. If $m_{4}=1$, then  either $m_{2}=2$, $m_{3}=1$ or $m_{2}=m_{3}=1$. Let $m_{2}=2$, $m_{3}=1$. Then by \cite{Bo}, there exist $x,y,z\in G$  such that $b_{1}=(x,y), ~b_{2}=(x,z), ~c_{1}=(x,y,y)$,  $G'= \langle b_{1},b_{2}, c_{1}, d_{1}\rangle$, $\gamma_{3}(G)= \langle c_{1}\rangle$ and $\gamma_{4}(G)=\langle d_{1}\rangle$
Therefore by  Lemma~\ref{L4_6}, $(b_{1}-1)^{2}(b_{2}-1)^{2}(c_{1}-1)^{2}(d_{1}-1)^{2}\in KG^{[14]}=0$,
a contradiction. 
Suppose $m_{2}=m_{3}=1$, then as in Lemma~\ref{Lemma 2.5}, $G'$ is abelian and $|G'|= 3^{3}$. If $m_{4}=0$, then $\gamma_{4}{(G)}=1$ and $G'$ is abelian.  Let  $(3^{m_{1}}, 3^{m_{2}}, \cdots, 3^{m_{s}})$ be the invariants of $G'$. Then by Theorem~\ref{thm_6}, $r+t\leq 12$,
where $ t= \sum\limits_{i=1}^{s}(3^{m_{i}}-1)$ is an even number and $\big\vert\gamma_{3}(G)G'^{3}/G'^{3}\big\vert = 3^{r}$. 
Clearly $t \geq 6$. We have the following cases:

 {\bf{Case 1.}}	If $t=12$, then $r=0$. Thus $G'\cong C_{9}\times C_{3}\times C_{3}$ or $(C_{3})^{6}$ and $\gamma_{3}(G)\subseteq G'^{3}$. Hence  $t_{L}(KG)= t^{L}(KG) =14$ by \cite[Theorem 3.2]{BKu}. 

{\bf Case 2}: Let $t=10$. Then $G'\cong C_{9}\times C_{3}$ or $(C_{3})^{5}$ and $r\leq 2$. If $r=0$, then $\gamma_{3}(G)\subseteq G'^{3}$ and by \cite[Theorem 3.2]{BKu}, $t_{L}(KG)= 12$. If $r=1$ and $G'\cong (C_{3})^{5}$, then $\gamma_{3}(G)\cong C_{3}$. Hence $d_{(2)}=4$, $d_{(3)}=1$ and $t^{L}(KG) =14$.  If $r=2$ and $G'\cong (C_{3})^{5}$, then $m_{2}=3$ and $m_{3}=2$, a contradiction by Lemma~\ref{L5_6}. Let $r=1$ and $G'\cong C_{9}\times C_{3}$. Then  $|D_{(3),K}(G)|=9$. Thus $|\gamma_{3}(G)|\leq9$. Let $\gamma_{4}(G)\cong C_{3}$ and $\gamma_{4}(G)\cap G'^{3}=1$. Then  $|D_{(4),K}(G)|=3^{2}$, $d_{(2)}=1$, $d_{(3)}=0$, $d_{(4)}=2$ which is not possible by \cite{Sh2}. So $ \gamma_{4}(G)\subseteq G'^{3}$, $|D_{(4),K}(G)|=3$, $d_{(2)}=d_{(3)}=d_{(4)}=1$ and $t^{L}(KG) =14$.
If $G'\cong C_{9}\times C_{3}$ and $r=2$, then  $|D_{(3),K}(G)|=3^{3}$ and $d_{(2)}=0$ which is not possible.
 
{\bf Case 3}: Let $t=8$. Then $G'\cong (C_{3})^{4}$ and $r\leq 3$.
 If $r=3$, then $\gamma_{3}(G)\cong (C_{3})^{3}$. Let $\gamma_{3}(G)=\langle c_{1},c_{2},c_{3} \rangle$ and let $b_{1}\in G'\backslash\gamma_{3}(G)$ such that $G'=\langle b_{1},c_{1},c_{2},c_{3} \rangle$. As in  \cite{Bo}, there exist $x,y\in G$ such that $b_{1}=(x,y)$ and  $c_{1}=(x,y,y)$. Then by Lemmas~\ref{L2_6} and \ref{L4_6}, $(b_{1}-1)^{2}(c_{1}-1)^{2}(c_{2}-1)^{2}(c_{3}-1)^{2}\in KG^{[14]}=0$,
which is a contradiction. If $r\leq 1$, then $t_{L}(KG)\leq t^{L}(KG)\leq12$. So $r=2$ and $\gamma_{3}(G)\cong C_{3}\times C_{3}$.  Hence $d_{(2)}= d_{(3)}=2$ and $t^{L}(KG) =14$.

{\bf{Case 4.}} If $t=6$, then $G'\cong C_{3}\times C_{3}\times C_{3}$ and $r\leq2$. If $r\leq 1$, then $t_{L}(KG)\leq t^{L}(KG)\leq10$.  So $r=2$ and  $D_{(3),K}(G)=\gamma_{3}(G)\cong C_{3}\times C_{3}$. If $\gamma_{4}(G)=1$, then $t_{L}(KG)\leq t^{L}(KG)=12$. So $\gamma_{4}(G)\cong C_{3}$,  $d_{(2)}=d_{(3)}=d_{(4)}=1$ and $t^{L}(KG) =14$.

Conversely, if $p=3$ and $t^{L}(KG)= 14$, then $t_{L}(KG)\leq 14$. We can not have $t_{L}(KG)=13$ by Lemma~\ref{Lemma 2.5}. Also $t_{L}(KG)\neq9, 10$ and $11$ by  \cite{Sa} and Lemma~\ref{Lemma 2.4}. For $t_{L}(KG)\leq 8$ yields $t^{L}(KG)=t_{L}(KG)\leq 8$, by \cite{CS,msbs5}. If $t_{L}(KG)=12$,  then $t^{L}(KG)=12$ by Theorem~\ref{thm1_6}. So $t_{L}(KG)=14$.
\end{proof}
\begin{remark}
Classification of Lie nilpotent group algebras $KG$   with the upper Lie nilpotency index $t^{L}(KG)=5p-3$ and $6p-4$ is given in \cite{msbs1}.
\end{remark}
{\bf{Acknowledgments}}: The financial assistance provided to the second author  from University Grants Commission, New Delhi, India is gratefully acknowledged.

\end{document}